\documentclass[reqno]{amsart}     
\RequirePackage{amsmath}
\RequirePackage{amsthm}
\RequirePackage{amssymb}
\RequirePackage{latexsym}
\RequirePackage[mathscr]{eucal}
\RequirePackage{verbatim}
\RequirePackage{enumerate}
\RequirePackage{mathrsfs}
\RequirePackage[colorlinks]{hyperref}
\RequirePackage{url}
   \usepackage{graphics}

\RequirePackage{amsmath}
\RequirePackage{amsthm}
\RequirePackage{amssymb}
\RequirePackage{latexsym}
\RequirePackage[mathscr]{eucal}
\RequirePackage{verbatim}
\RequirePackage{xspace}
\RequirePackage{mathrsfs}
\RequirePackage[all,cmtip]{xy}
\RequirePackage{hyperref}
\RequirePackage{amssymb}
\RequirePackage[english,algoruled,vlined,linesnumbered]{algorithm2e}

\usepackage{color}  

\newtheorem{theorem}{Theorem}

\newtheorem{lemma}[theorem]{Lemma}

\newtheorem{proposition}[theorem]{Proposition}

\newtheorem{claim}[theorem]{Claim}

\newtheorem{conj}[theorem]{Conjecture}

\let\eps=\varepsilon
\let\theta=\vartheta
\let\rho=\varrho
\let\sigma=\varsigma

\let\polishlcross=\l
\def\l{\ifmmode\ell\else\polishlcross\fi}

\def\dcup{\dot\cup}

\def\qand{\quad\text{and}\quad}

\newcommand{\supp}{\mathrm{supp}}

\newcommand{\cK}{\mathcal{K}}

\def\Nat{\mathbb N}

\def\pr{\phi_{\chi}^R}
\def\prn{\phi^R}

\def\nr{N_{\chi}^R}

\def\ccH{\mathcal{H}}

\newcommand{\ex}{\mathrm{ex}}

\newcommand{\edk}[1]{N(n,#1,K_r)}
\newcommand{\edt}[2]{N(n,#1,K_{#2})}

\newcommand{\PP}{\mathbb{P}}
\newcommand{\EE}{\mathbb{E}}

\author[L. \"Ozkahya]{Lale \"Ozkahya}
\address{Hacettepe University, Department of Computer Engineering, 
Beytepe 06800 Ankara, Turkey}
\email{ozkahya@hacettepe.edu.tr}

\author[Y.~Person]{Yury Person}
\address{Goethe-Universit\"at, Institut f\"ur Mathematik,
  Robert-Mayer-Str. 10, 60325 Frankfurt am Main, Germany}
\email{person@math.uni-frankfurt.de}
\thanks{YP was supported by DFG grant PE 2299/1-1.}

\begin{document}
\title[Minimum rainbow $H$-decompositions of graphs]
{Minimum rainbow $H$-decompositions of graphs}

\date{\today}
\begin{abstract}
Given graphs  $G$ and $H$, we consider the problem of decomposing a properly edge-colored graph $G$ into few parts consisting of rainbow copies of $H$ and single edges. We establish a close relation to the previously studied problem of minimum $H$-decompositions, where an edge coloring does not matter and one is merely interested in decomposing graphs into copies of $H$ and single edges.
\end{abstract}

\maketitle

\section{Introduction and new results}\label{sec:intro}
For two graphs $G$ and $H$ and a proper coloring $\chi$ of $G$, 
a \emph{rainbow $H$-decomposition} of $G$ is a partition  
of the edges of $G$ into $G_1$,\ldots, $G_t$ 
such that every $G_i$ is either a single edge 
or is rainbow-colored and isomorphic to $H$, where a rainbow coloring of $H$ assigns distinct colors to all edges of $H$. 
Rainbow $H$-decompositions present a variation on a well-studied topic of $H$-decompositions. Before we state our main results, we present the history of this general problem.

\subsection{Previous work on $H$-decompositions}

For two graphs $G$ and $H$, 
an \emph{$H$-decomposition} of $G$ is a partition  
of the edges of $G$ into $G_1$,\ldots, $G_t$ 
such that every $G_i$ is either a single edge 
or is isomorphic to $H$. 
An $H$-decomposition of $G$ with smallest possible $t$ is called 
\emph{minimum} and $\phi(G, H)=t$ denotes its cardinality. 
Let $N(G, H)$ denote the maximum number of edge-disjoint copies of $H$ in $G$, then 
the following relation $\phi(G, H)=e(G)-(e(H)-1)N(G, H)$ clearly holds.

We write $\phi(n,H):=\max_{G\in \mathcal{G}_n} \phi(G, H)$ for a general function, 
where $\mathcal{G}_n$ denotes the family of all graphs on $n$ vertices.
 This function was studied first by Erd{\H o}s, Goodman and P\'osa~\cite{EGP66}, 
who were motivated by the problem of representing graphs by set intersections. 
They showed that $\phi (n, K_3)=\ex(n,K_3)$, 
where $\ex(n,H)$ denotes the maximum size of a graph on $n$ vertices, 
that does not contain $H$ as a subgraph. 
Moreover, these authors  proved in~\cite{EGP66}  that the only graph that maximizes this function 
is the complete balanced bipartite graph. 
Consequently, they conjectured that $\phi (n, K_r)=\ex(n,K_r)$ and 
the only optimal graph is the Tur\'an graph $T_{r-1}(n)$, 
the complete balanced $(r-1)$-partite graph on $n$ vertices, 
where the sizes of the partite sets differ from each other by at most one. 
 Bollob\'as~\cite{Bo76} verified this conjecture by showing that $\phi (n,K_r)=\ex(n,K_r)$ for all  $n\ge r\ge 3$.

Pikhurko and Sousa~\cite{PS07} proved that for any fixed graph $H$ 
with chromatic number $r\ge 3$, $\phi (n, H)=\ex(n,H)+o(n^2)$ and they 
 made the following conjecture.
\begin{conj}\label{conj1}
 For any graph $H$ with chromatic number at least 3, there is an  $n_0=n_0(H)$ such that $\phi (n, H)=\ex(n,H)$ for all $n\ge n_0$.
\end{conj}
This conjecture has been verified by Sousa for 
clique extensions of order $r\ge 4$ ($n\ge r$)~\cite{sousa2010+}, 
the cycles of length $5$ ($n\ge 6$) 
and $7$ ($n\ge 10$)~\cite{sousa2005,sousa2011}.
In a previous work~\cite{LO-YP} we verified Conjecture~\ref{conj1} for \emph{edge-critical} graphs, 
where a graph $H$ is  edge-critical if there exists an edge $e\in E(H)$ such that $\chi(H)>\chi(H-e)$. 
Later Allen, B\"ottcher and Person~\cite{ABP14} obtained general upper 
bounds for all graphs $H$ that improve the error term in the result 
of Pikhurko and Sousa~\cite{PS07} and extend the result on the edge-critical case.

Recently, Liu and Sousa~\cite{LS14} studied the following colored variant of the $H$-decomposition problem. 
 Let $\phi_k(G,H)$ be  the smallest number $t$ such that any graph $G$ of order $n$ and 
any coloring of its edges with $k$ colors admits a monochromatic $H$-decomposition 
with at most $t$ parts, where $H$ is a complete graph $K_r$. 
Later, Liu, Pikhurko and Sousa~\cite{LPS15} generalized this by investigating 
$\phi_k(G,\mathcal{H})$, which is the smallest number $t$ such that any graph $G$ 
of order $n$ and any coloring of its edges with $k$ colors admits a 
monochromatic $\mathcal{H}=\{H_1,\dots, H_k\}$-decomposition 
such that each part
is either a single edge or forms a monochromatic copy of $H_i$ in color $i$,
for some $1\le i\le k$.  
Extending the results of Liu and Sousa~\cite{LS14}, 
they solve this problem when each graph in $\mathcal{H}$ is
a clique and $n\ge n_0(\mathcal{H})$ is sufficiently large.

\subsection{New results}

 We call a rainbow $H$-decomposition of $G$ under the coloring $\chi$ 
with the smallest possible $t$
\emph{minimum} and $\pr(G,H)=t$ denotes its cardinality. 
It is not difficult to see that 
$\pr(G, H)=e(G)-(e(H)-1)\nr(G,H)$, where 
$\nr(G, H)$ denotes the maximum number of edge-disjoint rainbow copies of $H$ under this   
proper coloring $\chi$ of $G$.

In this paper, we study the function
\[
 \prn(n,H):=\max_{G\in \mathcal{G}_n} \max_{\chi}\pr(G, H),
\]
where we maximize over all graphs $G$ from  $\mathcal{G}_n$ and  over all proper colorings 
 $\chi$  of $G$. We will refer to decompositions that attain $\prn(n,H)$
  as \emph{minimum rainbow $H$-decompositions of graphs}.

Observe that $ \pr(G, H) \geq \phi(G,H)$ for any graph $G$ and any proper edge coloring $\chi$ of $G$, 
otherwise it means that there are more than $N(G, H)$ edge-disjoint rainbow 
copies of $H$ in  $G$ under $\chi$, which is a contradiction.  
Therefore, we have 
$\prn(n,H)\geq \phi(n, H)$.
On the other hand, $\max_{\chi}\pr(G, H)\leq e(G)$, where 
equality is achieved when there is no copy of $H$ in $G$, otherwise 
one can always color $G$ properly so that a particular copy of $H$ 
is rainbow. 

We prove the following result for any clique $K_r$, $r\ge 3$. 

\begin{theorem}\label{thm:main-rainbow}
For any $r\ge 3$  
there is an $n_0$ such that any graph $G$ on $n\ge n_0$ vertices with some proper edge coloring  $\chi$ that satisfies $\pr(G,K_r)\ge \ex(n,K_r)$ must in fact be isomorphic to the Tur\'an graph $T_{r-1}(n)$.

In particular, $\prn(n,K_r)=\ex(n,K_r)$ for all $n\ge n_0$, and  the only graph attaining $\prn(n,K_r)$ is the Tur\'an graph $T_{r-1}(n)$.
\end{theorem}  

We also obtain generalizations of the result of Pikhurko and Sousa~\cite{PS07} 
on $\phi (n, H)=\ex(n,H)+o(n^2)$ (for non-bipartite $H$) and on our result from~\cite{LO-YP} about 
$\phi (n, H)=\ex(n,H)$ for edge-critical graphs $H$. Since we provide only very rough sketches of these generalizations, we postpone their discussion to the concluding remarks section, Section~\ref{sec:remarks}. 

Our proofs will combine stability approach with probabilistic techniques.
The paper is organized as follows. In the section below, Section~\ref{sec:tools}, we collect the various results that we are going to use. 
In Section~\ref{sec:stability}, we prove a new stability result about the function $\phi(G,K_r)$ and building on that, we show various stability results about the function $\pr(G,K_r)$. Our proof of the rainbow stability for $\pr(G,K_r)$, Theorem~\ref{thm:stab_phi_rainbow}, is a main contribution of this paper.  
 In Section~~\ref{sec:exact} we provide the (sketch of the) proof of Theorem~\ref{thm:main-rainbow}. 
Finally, we explain in Section~\ref{sec:remarks} how the
 exact version for edge-critical $H$  and an approximate version for nonbipartite $H$ of the function $\prn(n,H)$ follow.

\subsection{Notation}
Throughout the sections, we omit floor and ceiling notations, since they do not 
affect our calculations.   
We use standard notations from graph theory. 
Thus, for $t\in\Nat$ we denote by $[t]$ the set $\{1,\ldots,t\}$. 
For a given graph $G=(V,E)$ and for a subset $U\subseteq V$ 
we denote by $E_G(U)=E\cap\binom{U}{2}$ and $G[U]=G(U,E_G(U))$. 
We set $e_G(U)=|E_G(U)|$, and for a vertex $v\in V$ 
we write $\deg_{G,U}(v)=|\{u\in U\colon \{v,u\}\in E(G)\}|$, 
i.e., we are only counting the neighbors of $v$ in $U$. 
Similarly, for two disjoint subsets $U,W\subseteq V$ 
we set $E_G(U,W)=\{\{u,w\}\in E(G)\colon u\in U, w\in W\}$, 
$G[U,W]=G(U\cup W, E_G(U,W))$ and 
$e_G(U,W)=|E_G(U,W)|$. 
We will sometimes omit $G$ 
when there is no danger of confusion, 
and we write $\deg_U(v)$, $e(U)$, $E(U,W)$, $e(U,W)$.

\section{Tools}\label{sec:tools}
\subsection{Probabilistic tools}
We will make use of the following version of Chernoff's inequality, see e.g.~\cite[Corollary~2.4 and Theorem~2.8]{JLR}.
\begin{theorem}[Chernoff's inequality]\label{thm:Chernoff}
 Let $X$ be the sum of independent binomial random variables, then for any $\delta \in(0,3/2]$ we have
\begin{align*}
 \PP [|X  -\EE(X) |\ge \delta \EE(X)] \le 2\exp (-\delta^2 \EE (X)/3).
\end{align*}
Moreover, we have
\[
\PP[X\ge x]\le \exp(-x)\text{  for  }x\ge 7\EE(X).
\]
\end{theorem}
Another concentration result that we are going to employ is a theorem due to Kim and Vu~\cite{KimVu00}.  We state it in a slightly less general version (without weights on the edges) suited for our purposes.  
\begin{theorem}[Kim-Vu polynomial concentration result~\cite{KimVu00}]\label{thm:kimvu}
Let $H=(V,E)$ be a (not necessarily uniform) hypergraph on $n$ vertices whose edges have cardinality at most $k\in\Nat$, let $(X_v)_{v\in V}$ be a family of mutually independent binomial random variables and set $X:=\sum_{e\in E}\prod_{v\in e}X_v$. Then, for any $\lambda > 1$, we have 
\[
\PP\left[\left|X-\EE(X)\right|>a_k (EE')^{1/2}\lambda^k\right]<d_ke^{-\lambda}n^{k-1},
\]
where $a_k=8^kk!^{1/2}$, $d_k=2e^2$, $n=|V|$, $E=\max(\EE(X),E')$ and $E'=\max_{i=1}^k E_i$, and the quantity $E_i$ is defined as follows
\[
E_i:=\max_{A\subseteq V, |A|=i} \EE\left(\sum_{e\in E, A\subseteq e} \prod_{v\in e\setminus A}X_v\right).
\]
\end{theorem}

\subsection{Extremal graph theoretic results}
Let $\edk{m}$ denote the minimum of $N(G,K_r)$ over all graphs 
$G$ with $n$ vertices and $m$ edges and recall that $N(G, K_r)$ denotes the maximum number of edge-disjoint copies of $K_r$ in $G$. Below we summarize the lower bounds on the function $\edk{m}$ shown by Gy\H{o}ri and Tuza~\cite{GT87} and by Hoi~\cite{Hoi05} which we are going to use.

\begin{theorem}\label{thm:edk}
The following bounds hold:
\begin{enumerate}[(i)]
\item $\edt{\ex(n,K_3)+m}{3}\ge \left(\tfrac{5}{9}+o(1)\right)m$ (Gy\H{o}ri and Tuza~\cite{GT87}), and 
\item $\edk{\ex(n,K_r)+m}\ge \frac{m}{\binom{r}{2}-(r-2)}$ (Hoi~\cite[Theorem~1.1]{Hoi05}).
\end{enumerate}
\end{theorem}
We remark, that for particular $r$ better bounds are known ($\edt{\ex(n,K_4)+m}{4}\ge \frac{5m}{18}-o(n^2)$ from~\cite{Hoi05}) and in the case $m=o(n^2)$ as well, see Gy\H{o}ri~\cite{G91}. We will combine the theorem above with the classical stability result due to Erd{\H o}s~\cite{Erd68} and Simonovits~\cite{Si68}. 
\begin{theorem}[Stability theorem]\label{thm:Simstability}
 For every $H$ with  $\chi(H)=r\ge 3$, and every $\gamma>0$ there exist a $\delta>0$ and an $n_0$ such that the following holds. If $G$ is a graph on $n\ge n_0$ vertices with $e(G)\ge \ex(n,H)-\delta n^2$ and 
if it does not contain $H$ as a subgraph, 
then there exists a partition of $V(G)=V_1\dcup\ldots,\dcup V_{r-1}$ such that $\sum_{i=1}^{r-1} e(V_i)< \gamma n^2$.
\end{theorem}

Another versatile tool that we are  going to apply is the regularity lemma of Szemer\'edi~\cite{Sz78}. Before stating it,  we introduce first the central concepts. We say that a bipartite graph $G=(V_1\dcup V_2,E)$, or simply $(V_1,V_2)$, is $\eps$-regular  if all pairs of subsets  $U_i\subseteq V_i$, with $|U_i|\ge\eps|V_i|$, $i=1,2$, satisfy
\[
|d_G(V_1,V_2)-d_G(U_1,U_2)|\leq \eps,
\]
where $d_G(U_1,U_2):=\frac{e_G(U_1,U_2)}{|U_1||U_2|}$ is the density of the bipartite graph induced by the color classes $U_1$ and $U_2$. 
An $\eps$-regular pair $(V_1,V_2)$ is called $(\eps,d)$-regular if it has density at least $d$. 

Now consider a partition $\{V_1,\ldots,V_t\}$ of $V$ such that $|V_1|\le|V_2|\le\ldots\le|V_t|\le|V_1|+1$. We call such partition equitable. We refer to an equitable  partition as $\eps$-\emph{regular} if  it satisfies the condition that all but $\eps\binom{t}{2}$ pairs $(V_i,V_j)$ are $\eps$-regular, where $i< j\in [t]$. The vertex subsets $V_i$ are referred to as \emph{clusters} or \emph{classes}.  

The regularity lemma states then the following.
\begin{theorem}[Regularity lemma]\label{thm:reglemma}
For every integer $t_0\ge 1$ and every $\eps>0$ there exist integers $T_0=T_0(t_0,\eps)$
and  $n_0=n_0(t_0,\eps)$ such that every graph $G=(V,E)$ on at least $n_0$ vertices admits an $\eps$-regular partition $V=V_1\dcup\ldots\dcup V_t$ with $t_0\le t \le T_0$.
\end{theorem}

The regularity lemma is accompanied by a very useful fact called the \emph{counting lemma}, see e.g.\ a survey of  Koml\'os and Simonovits~\cite{KoSi96}. We state here one version that will be enough for our needs.
\begin{lemma}[Counting lemma]\label{lem:counting}
For every $r\ge3$, every $\delta>0 $ and $\gamma>0$ there exist an $\eps>0$ and $m_0$ such that the following holds. Let $m\ge m_0$, let $V_1$, \ldots, $V_r$ be vertex-disjoint subsets of size $m$ or $m+1$ of some graph $G$, such that each pair $(V_i,V_j)$ (for $i\neq j\in[r]$) is $\eps$-regular and has density $d_{ij}\ge \delta$. Then, for all $i\neq j$, all but at most $4r\eps (m+1)^2$ edges from $(V_i,V_j)$ lie in 
\begin{equation}\label{eq:count}
\left(1\pm\gamma\right)\prod_{\{k,\ell\}\in\binom{[r]}{2}, \{k,\ell\}\neq\{i,j\}} d_{k\ell}\prod_{s\in[r],s\neq i,j}|V_s|
\end{equation}
 copies $K$ of $K_r$ such that $|K\cap V_i|=1$ for all $i$.
\end{lemma}
We will also need the following theorem of Pippenger and Spencer~\cite{PS89}, 
see also R\"odl~\cite{VR85} and Alon and Spencer~[Theorem~4.7.1]\cite{AS00}.  
\begin{theorem}\label{PS-matching}
For every integer $s\ge 2$, $k\ge 1$ and real $c_0>0$, there are $c_1=c_1(s,c_0, k)>0$ and 
$d_0 = d_0(s, c_0, k)$ such that for any $n$ and $D\geq d_0$ the following holds.

Every $s$-uniform hypergraph ${\mathcal H}$ on a set $V$ of $n$ vertices 
satisfying all of the following conditions\\
(1) for all $x\in V$ but at most $c_1n$ of them, 
${\rm deg}(x) = (1\pm c_1)D$; \\
(2) for all $x\in V$, ${\rm deg}(x)\leq kD$;\\
(3) for any two distinct $x, y\in V$, ${\rm codeg}(x, y)<c_1 D$;\\
contains a matching consisting of at least $(1-c_0)n/s$ hyperedges.  
\end{theorem}

\section{Stability results for rainbow $K_r$-decompositions}\label{sec:stability}
\subsection{Warm-up: stability for $\phi(G,K_r)$}

In~\cite{LO-YP}  we proved the following approximate result about graphs 
$G\in\mathcal{G}_n$ with $\phi(G,H)\ge\ex(n,H)-o(n^2)$. 
\begin{theorem}[Lemma~4 from~\cite{LO-YP}]\label{lem:stability}
 For every $H$ with $\chi(H)=r\ge 3$, $H\neq K_r$, and for every $\gamma>0$ there exist $\eps>0$ and $n_0\in\Nat$ such that for every graph $G$ on $n\ge n_0$ vertices the following is true. If 
\[
\phi(G,H)\ge\ex(n,H)-\eps n^2 
\]
then there exists a partition of $V(G)=V_1\dcup\ldots\dcup V_{r-1}$ with $\sum_{i=1}^{r-1}e(V_i)<\gamma n^2$.
\end{theorem}
Our proof was built on a theorem of Pikhurko and Sousa~\cite{PS07} about weighted decompositions of graphs. Due to a technical calculation, the natural case $H=K_r$ remained uncovered. The following proposition provides a short proof of the stability for cliques for the function $\phi(G,K_r)$. It will be used to obtain  the stability for cliques for the rainbow function $\pr(G,K_r)$. 

\begin{proposition}\label{thm:stab_phi_cliques}
 For every $r\ge 3$ and for every $\gamma>0$ there exist $\eps>0$ and $n_0\in\Nat$ such that for every graph $G$ on $n\ge n_0$ vertices the following is true. If 
\[
\phi(G,K_r)\ge\ex(n,K_r)-\eps n^2 
\]
then there exists a partition of $V(G)=V_1\dcup\ldots\dcup V_{r-1}$ with $\sum_{i=1}^{r-1}e(V_i)<\gamma n^2$.
\end{proposition}
\begin{proof} Let $\gamma> 0 $ and $r\ge 3$ be given. Take $\delta$ and $n_0$ as guaranteed by Theorem~\ref{thm:Simstability} on input $\gamma$ and $K_r$, and choose with foresight $\eps:=\frac{\min(\gamma/2,\delta)}{10r^3}$. 
Let $G$ be a graph on $n\ge n_0$ vertices with $\phi(G,K_r)\ge\ex(n,K_r)-\eps n^2$. Assume that $e(G)=\ex(n,K_r)+m$, where $m$ might be negative. It follows from the lower bound on $\phi(G,K_r)$ and from the identity $\phi(G,K_r)=e(G)-(\binom{r}{2}-1)N(G, K_r)$, that  
\begin{equation}\label{eq:NGKr_lowerbd}
N(G, K_r)\le \frac{m+\eps n^2}{\binom{r}{2}-1}.
\end{equation}
On the other hand we have for $m>0$, by Theorem~\ref{thm:edk}, that 
\begin{align*}
N(G,K_3)&\ge\edt{\ex(n,K_3)+m}{3}\ge \frac{(5+o(1)) m}{9}, \qand\\
 N(G,K_r)&\ge\edk{\ex(n,K_r)+m}\ge \frac{m}{\binom{r}{2}-(r-2)},
\end{align*}
and thus it follows with~\eqref{eq:NGKr_lowerbd} that $m\le 10r \eps n^2$ (if $m<0$ then we can use~\eqref{eq:NGKr_lowerbd} directly). Therefore we obtain the following bound on $N(G,K_r)$:
\[
N(G, K_r)\le \frac{10r \eps n^2+\eps n^2}{\binom{r}{2}-1}\le 10r \eps n^2.
\]
We delete  $\binom{r}{2}N(G, K_r)$ edges from $G$ making it $K_r$-free. Denote this new graph by $G'$. Clearly,
$e(G')\ge \ex(n,K_r)-\eps n^2-10r\binom{r}{2}\eps n^2$. Now we can apply Theorem~\ref{thm:Simstability} to $G'$ to obtain the desired partition. 
\end{proof}

\subsection{Rainbow stability}
The aim of this section is to prove auxiliary tools that will imply the following  stability result for minimum rainbow $K_r$-decompo\-sitions.
\begin{theorem}\label{thm:stab_phi_rainbow}
 For every $r\ge 3$ and for every $\gamma>0$ there exist $\eps>0$ and $n_0\in\Nat$ such that for every graph $G$ on $n\ge n_0$ vertices with a proper coloring $\chi$ of its edges the following is true. If 
\[
\pr(G,K_r)\ge\ex(n,K_r)-\eps n^2 
\]
then there exists a partition of $V(G)=V_1\dcup\ldots\dcup V_{r-1}$ with $\sum_{i=1}^{r-1}e(V_i)<\gamma n^2$.
\end{theorem}
Notice that the case $r=3$ is covered by Proposition~\ref{thm:stab_phi_cliques}, as any proper edge coloring colors a triangle with different colors. 

\subsubsection{Proof overview}
As already mentioned, Theorem~\ref{thm:stab_phi_rainbow} is a main contribution of this paper. 
 We fix a collection of edge-disjoint (not necessarily rainbow) copies of $K_r$ of maximum cardinality. The proof proceeds by an application of the regularity lemma of Szemer\'edi~\cite{Sz78}. Then we identify various $\eps$-regular pairs of not too small density that build $r$-partite graphs, where most edge-disjoint  copies of $K_r$ `live'. We need to argue then that we can find roughly that many rainbow copies instead. To do so, we would like to apply 
a Theorem of Pippenger and Spencer~\cite{PS89} (see also Alon and Spencer~[Theorem~4.7.1]\cite{AS00}) to decompose `most' of the edges into edge-disjoint rainbow copies. Before that we appropriately split $\eps$-regular pairs, similarly as is done for example in~\cite{PS07}. This time however, the density of these `split' $\eps$-regular pairs may be much lower than $\eps$, which would make the usual approach 
to apply the counting lemma (and then to apply a Theorem of Pippenger and Spencer~\cite{PS89})   impossible. Our solution is to count the rainbow copies of $K_r$ (via Lemma~\ref{lem:most_rainbow}) before splitting the $\eps$-regular pairs (of sufficiently high density) and only then to split them  randomly. All we need to estimate is then the number of rainbow copies that will be in (random) subgraphs of $\eps$-regular $r$-partite graphs, which we do by applying a concentration result of Kim and Vu~\cite{KimVu00} (more precisely, Lemma~\ref{lem:sparsify} above, which follows by an application of the result from Kim and Vu~\cite{KimVu00}).


\subsubsection{Some auxiliary results}

The following lemma asserts that, in a proper edge coloring,  for most of the edges from an ``$\eps$-regular environment'' most of the cliques $K_r$ that ``sit'' on them are in fact rainbow.
\begin{lemma}\label{lem:most_rainbow}
For every $r\ge3$, every $\delta>0 $ and $\gamma>0$ there exist an $\eps>0$ and $m_0$ such that the following holds. Let $m\ge m_0$, let $V_1$, \ldots, $V_r$ be vertex-disjoint subsets of size $m$ or $m+1$ of some graph $G$, such that each pair $(V_i,V_j)$ (for $i\neq j\in[r]$) is $\eps$-regular and has density $d_{ij}\ge \delta$. Further let $\chi$ be any proper edge coloring of $G$. Then, for all $i\neq j$, all but at most $4r\eps(m+1)^2$ edges from $(V_i,V_j)$ lie in 
\begin{equation}\label{eq:rainbow_count}
(1\pm\gamma)\left(\prod_{\substack{\{k,\ell\}\in\binom{[r]}{2}\\
\{k,\ell\}\neq\{i,j\}}} d_{k\ell}\right)\prod_{s\in[r],s\neq i,j}|V_s|
\end{equation}
rainbow copies of $K_r$ between the sets $V_1$,\ldots, $V_r$.
\end{lemma}
\begin{proof}Let $\eps>0$ and $m_0$ be asserted by the counting lemma, Lemma~\ref{lem:counting}, for supplied parameters $r$, $\delta$ and $\gamma/2$ (instead of $\gamma$). We clearly may assume that $r\ge 4$.

 Let $G$ be a graph and let $V_1$, \ldots, $V_r$ be its vertex-disjoint subsets of size $m$ or $m+1$. 
Then, clearly, it holds that all but at most $4r\eps(m+1)^2$ edges from $(V_i,V_j)$ lie in 
\begin{equation}\label{eq:counting}
(1\pm\gamma/2)\left(\prod_{\substack{\{k,\ell\}\in\binom{[r]}{2}\\
\{k,\ell\}\neq\{i,j\}}} d_{k\ell}\right)\prod_{s\in[r],s\neq i,j}|V_s|
\end{equation}
 many copies of $K_r$ between the sets $V_1$,\ldots,$V_r$. We refer to such edges as \emph{good}.

We fix an arbitrary proper edge coloring $\chi$ of $G$.  Let $v_iv_j\in E(V_i,V_j)$ be a good edge and let $K$ be a copy of $K_r$ between $V_1$,\ldots, $V_r$, which contains $v_iv_j$. 
Assume that $K$ is not rainbow and let the vertex set of $K$ be $\{v_1,\ldots,v_r\}$. Then there exist four distinct  indices $s_1$,\ldots, $s_4\in[r]$ with $\chi(v_{s_1}v_{s_2})=\chi(v_{s_3}v_{s_4})$. In the following we estimate the number of such non-rainbow copies of $K_r$, where we will make use of the fact that in a proper edge coloring every color class forms a matching and there are thus at most $r(m+1)/2$ edges of the same color in $G$. We distinguish three cases:
\begin{enumerate}
\item one of the edges $v_{s_1}v_{s_2}$, $v_{s_3}v_{s_4}$ (say $v_{s_1}v_{s_2}$) equals to $v_iv_j$; 
there are at most $r(m+1)/2$ choices for the edge $v_{s_3}v_{s_4}$ and, therefore, at most $\tfrac{r(m+1)}{2}(m+1)^{r-4}=o(m^{r-2})$ such non-rainbow copies $K$;
\item at least one of the edges $v_{s_1}v_{s_2}$, $v_{s_3}v_{s_4}$ (say $v_{s_1}v_{s_2}$) 
is incident to $v_i$ or $v_j$; then as $v_i$ and $v_j$ each has  at most $(r-2)(m+1)$ neighbours among $V_\ell$s ($\ell\neq i,j$) this number is an upper bound on the number of possible colors with $\chi(v_{s_1}v_{s_2})=\chi(v_{s_3}v_{s_4})$. Therefore, there are at most $(r-2)(m+1)(m+1)^{r-4}=o(m^{r-2})$ non-rainbow copies $K$ where both edges 
$v_{s_1}v_{s_2}$, $v_{s_3}v_{s_4}$ are incident to $v_i$, $v_j$ respectively. Moreover, the number of non-rainbow copies $K$ with $\{v_{s_3},v_{s_4}\}\cap\{v_i,v_j\}=\emptyset$ is at most $(r-2)(m+1) \tfrac{r(m+1)}{2} (m+1)^{r-5}=o(m^{r-2})$. Summing up, this gives at most $o(m^{r-2})$ non-rainbow copies in this case.
\item both edges $v_{s_1}v_{s_2}$, $v_{s_3}v_{s_4}$ are disjoint from $\{v_i,v_j\}$. We write $m_i$ for the number of edges 
from $G[V_1\ldots V_r]$ in color $i$ and observe that $m_i\le \tfrac{r(m+1)}{2}$. The  number of non-rainbow copies $K$ is in this case at most 
\[
\sum_{i} m_i^2 m^{r-6}\le \tfrac{r(m+1)}{2} m^{r-6}\sum_i m_i\le \tfrac{r(m+1)}{2} m^{r-6}\binom{m(r+1)}{2}=o(m^{r-2}).
\]
\end{enumerate}
Thus, it follows that the total number of non-rainbow copies of $K_r$ that contain $v_iv_j$ is $o(m^{r-2})$. Since the number of copies of $K_r$ is given by~\eqref{eq:counting} and each $|V_\ell|\in\{m,m+1\}$ we 
immediately infer~\eqref{eq:rainbow_count} for $m_0$ sufficiently large.
\end{proof}

Later in our proof we will ``sparsify'' randomly our $\eps$-regular pairs such that the densities of these pairs are much below $\eps$.  As a consequence, we will not be able to count within these sparse pairs directly (as the density might be much less than the regularity parameter $\eps$). Instead we count before sparsification and the following lemma asserts that as many  rainbow copies remain as we would expect. Given an $r$-partite graph $G$ with the classes $V_1$,\ldots,$V_r$, and $p_{ij}\in[0,1]$ for all $\{i,j\}\in\binom{[r]}{2}$, we denote by $G_{(p_{ij})}$ the random subgraph of $G$ where each edge from $(V_i,V_j)$ is included with probability $p_{ij}$ independently of the other edges. 

Let 
$v_sv_t$ be an edge from 
$G[V_s,V_t]$ and let $K$ be a copy of 
$K_r$ that contains 
$v_sv_t$. We then say that the edge 
$v_sv_t$ {\it closes} a copy of $K$ in $G_{(p_{ij})}$ if all edges of 
$E(K)\setminus\{v_sv_t\}$ lie in $G_{(p_{ij})}$.
\begin{lemma}\label{lem:sparsify}
 For $r\ge 3$, $\gamma>0$ and $\eta>0$ there exists a $\beta>0$ and $m_0$ such that the following holds. Let $G$ be an $r$-partite  graph with classes $V_1$,\ldots,$V_r$, each of size $m$ or $m+1$, where $m\ge m_0$. Let $v_1v_2$ be an edge from $(V_1,V_2)$ and let $\cK$ be a family of some $\delta m^{r-2}$ copies of $K_r$ in $G$ that contain $v_1v_2$. Then for any
  sequence of  $p_{ij}>0$ ($i\neq j\in[r]$), such that 
\[  
\delta m^{r-2}\prod_{\substack{\{k,\ell\}\in\binom{[r]}{2}, \\
\{k,\ell\}\neq\{1,2\}}} p_{k\ell}\ge (m+1)^{r-3+\gamma},
\] 
it holds that the probability that the edge $v_1v_2$ does not close in $G_{(p_{ij})}$ 
  \[
   (1\pm \eta)\left(\prod_{\substack{\{k,\ell\}\in\binom{[r]}{2}\\ 
\{k,\ell\}\neq\{1,2\}}} p_{k\ell}\right) \delta m^{r-2}
  \] 
many copies of $K_r$  from $\cK$ is at most $\exp(-m^{\beta})$.
\end{lemma}
\begin{proof}
For every edge $e\in E(G)$ let $X_e$ be the indicator random variable whether the edge $e$ is in $G_{(p_{ij})}$. 
By the definition, we have $\PP(e\in G_{(p_{ij})})=p_{ij}$ for $e\in E_G(V_i,V_j)$. Further we define the random variable 
$X$ that counts the number of copies $K$ of $K_r$ from $\cK$,  
that are closed by $v_1v_2$, as follows: 
\[
X:=\sum_{\substack{
K \in \cK\\v_1v_2\in K}} 
\prod_{\substack{e\in E(K)\\
e\neq \{v_1v_2\}}} X_e.
\]
We clearly have 
\begin{equation}\label{eq:low_EE}
\EE(X)=|\cK| \prod_{\substack{\{k,\ell\}\in\binom{[r]}{2}\\
\{k,\ell\}\neq\{1,2\}}} p_{k\ell}\ge (m+1)^{r-3+\gamma},
\end{equation}
where the second inequality is an assumption of the lemma. We set $k:=\binom{r}{2}-1$ and 
 we simply estimate the quatity
\[
E':=\max_{i\in[k]}\max_{\substack{A\subseteq E(G)\\
|A|=i}} \EE\left(\sum_{\substack{K\in \cK\\
v_1v_2\in K\\
A\subseteq E(K)\setminus\{v_1v_2\}}} \prod_{e\in E(K)\setminus (A\cup\{v_1v_2\})}X_e\right)
\]
 by $(m+1)^{r-3}$, since by choosing some edge $f\neq v_1v_2$ there remain at most 
$(m+1)^{r-3}$ possible copies of $K_r$ in $G$ that contain $v_1v_2$ and $f$ as edges.

We may now apply Kim-Vu polynomial concentration (Theorem~\ref{thm:kimvu}) with $E=\max(\EE(X),E')=\EE(X)$ and $E'=(m+1)^{r-3}$ as follows 
(with $k=\binom{r}{2}-1$, $a_k=8^kk!^{1/2}$, $d_k=2e^2$):
\[
\PP\left[\left|X-\EE(X)\right|>a_k (EE')^{1/2}\lambda^k\right]<d_ke^{-\lambda} \left(\binom{r}{2}(m+1)^2\right)^{k-1}.
\]
By choosing $\lambda=\left(\frac{\eta \EE(X)}{a_k (EE')^{1/2}}\right)^{1/k}$ we have 
$\PP\left[\left|X-\EE(X)\right|>\eta \EE(X)\right]<d_ke^{-\lambda}(\binom{r}{2}(m+1)^2)^{k-1}$. 
Using~\eqref{eq:low_EE} we obtain the following lower bound on $\lambda$, which will be sufficient for our purposes:
\[
\lambda\ge \left(\frac{\eta (m+1)^{(r-3+\gamma)/2}}{8^kk!^{1/2} (m+1)^{(r-3)/2}}\right)^{1/k}\ge m^{\gamma/(3k)}.
\]
Thus, we may estimate the probability $\PP\left[\left|X-\EE(X)\right|>\eta \EE(X)\right]<e^{-m^{\gamma/(4k)}}$, 
setting $\beta=\gamma/(4k)$ and choosing $m_0$ sufficiently large.
\end{proof}

\subsubsection{Main lemma}
The following lemma shows that, for any graph $G$ on $n$ vertices and any proper edge coloring $\chi$, the numbers (of edge-disjoint copies) $N(G,K_r)$ and $\nr(G,K_r)$ differ by at most $o(n^2)$.

\begin{lemma}\label{lem:almost_edge_rainbow}
For all $r\ge3$ and any $c>0$ there exists an $n_0$ such that the following holds. In any proper edge coloring $\chi$ of a graph $G$ on $n\ge n_0$ vertices we have $N(G,K_r)\le\nr(G,K_r)+cn^2$.
\end{lemma}
\begin{proof}
Let $G$ be given, let $\chi$ be some fixed proper edge coloring of $G$ and let $\cK$ be a family of $N(G,K_r)$ edge-disjoint copies (not necessarily rainbow) of $K_r$ in $G$.

We set $\delta=c/3$ and $\xi=\delta/4$. We then choose $c_0=\xi$, $s=\binom{r}{2}$ and $k=2\delta^{-\binom{r}{2}+1}$, and let  $c_1=c_1(s,c_0, k)>0$ and $d_0 = d_0(s, c_0, k)$  be as asserted by Theorem~\ref{PS-matching} on input $s$, $k$ and $c_0$.  
 We set $\gamma=c_1/2$ and let $\eps'$ be as asserted by Lemma~\ref{lem:most_rainbow} on input $r$, $\delta$ and $\gamma$. Finally we choose $\eps:=\min\{\eps',\delta/4, \delta c_1/(63r)\}$ and $t_0=2/\eps$. Finally, let $T$ be as asserted by the regularity lemma (Theorem~\ref{thm:reglemma}) on input $\eps$ and $t_0$. We will also assume throughout the proof that $n$ is sufficiently large, so that all asymptotic estimates hold.

\textbf{\emph{An application of the regularity lemma.}} We apply regularity lemma to $G$ with (carefully chosen) parameters $\eps>0$ and $t_0$ (lower bound on the number of clusters). We obtain an $\eps$-regular partition of $V(G)$ into $V_1$,\ldots, $V_t$ where $t\le T=T(\eps,t_0)$. We define a cluster graph $R$ with the vertex set $\{V_i\colon i\in[t]\}$, where $ij\in E({R})$ whenever the density $d(V_i,V_j)\ge \delta$.  For convenience, we let $d_{ij}$ denote $d(V_i,V_j)$ for the remainder of the proof. First observe that the number of copies of $K_r$ from $\cK$ with at least one edge not from the pairs $(V_i,V_j)$ with $ij\in E({R})$ is at most
\begin{multline*}
\sum_{i=1}^t\binom{|V_i|}{2}+\sum_{ij\not\in E({R}), (V_i,V_j)\text{ is not }\eps\text{-regular}}|V_i||V_j|+
\sum_{ij\not\in E({R}), d_{ij}<\delta}e(V_i,V_j)\le\\
n^2/t_0+\eps n^2+\delta n^2/2<\delta n^2.
\end{multline*}
Thus, all but at most $\delta n^2$ copies from $\cK$ are completely contained within $\eps$-regular pairs of density at least $\delta$, and we denote such set of copies by $\cK'\subseteq \cK$ and identify these copies with their vertex sets.

\textbf{\emph{Calculating $\alpha_S$.}} 
Observe now that a copy of $K_r$ from $\cK'$ must lie between some $r$ clusters that form a copy of $K_r$ in $R$. We write $\supp(T)$ for $\{i\colon V_i\cap T\neq\emptyset\}$, the \emph{support} of $T$. For $S\in \binom{V({R})}{r}$, we denote by $\cK'(S)$ those copies $F$ of $K_r$ from  $\cK'$ with  $S=\supp(V(K))$. Thus, $|\cK'|=\sum_{S\in\binom{V({R})}{r}} |\cK'(S)|$. Of course, if the graph $R[S]$ is not complete for some $S$ then $|\cK'(S)|=0$.
 
For an $r$-element set   $S\subseteq V({R})$ with at least one copy $F\in \cK'$ with $\supp(F)=S$ we define the weight $\alpha_S$ as follows (for other $r$-element sets $S$ we set $\alpha_S=0$): 
\begin{equation}\label{eq:def_alpha_S}
\alpha_S:=\min_{\{i,j\}\in \binom{S}{2}}\frac{d_{ij}|\cK'(S)|}{|\{F\colon F\in\cK', i,j\in \supp(V(F)) \}|}.
\end{equation}

 Observe that for every $ij\in E({R})$ and $S\in\binom{V({R})}{r}$ with $S\supseteq \{i,j\}$ we have $\alpha_S\le d_{ij}$. Consider now an arbitrary $ij\in E({R})$ with $ |\{F\colon F\in\cK', i,j\in \supp(V(F)) \}|\neq 0$. From the equation 
 \begin{equation}\label{eq:split_cliques}
 |\{F\colon F\in\cK', i,j\in \supp(V(F)) \}|=\sum_{S\in\binom{V({R})}{r}\colon S\supseteq \{i,j\}}|\cK'(S)|
 \end{equation}
 we infer
 \[
 \sum_{S\colon S\supseteq \{i,j\}}\alpha_S\le 
 \sum_{S\colon S\supseteq \{i,j\}}\frac{d_{ij}|\cK'(S)|}{|\{F\colon F\in\cK', i,j\in \supp(V(F)) \}|}= 
 d_{ij}.
 \] 
 Furthermore, every copy from $\cK'$ whose support contains $i$ and $j$ uses exactly one distinct edge from 
 $E(V_i,V_j)$ and thus $|\{F\colon F\in\cK', i,j\in \supp(V(F)) \}|\le d_{ij} |V_i||V_j|$. We let $i$ and $j$ be the elements of $S$ that yield the value of $\alpha_S$ in~\eqref{eq:def_alpha_S} and conclude with~\eqref{eq:split_cliques}:
 \begin{equation}\label{eq:bound_alpha_cK}
 |\cK'(S)|= \frac{\alpha_S}{d_{ij}} |\{F\colon F\in\cK', i,j\in \supp(V(F)) \}|\le \alpha_S |V_i||V_j|.
 \end{equation}
Roughly speaking, the values $\alpha_S$ will tell us below where we have to look for \emph{many} 
edge-disjoint rainbow copies of $K_r$.
 

\textbf{\emph{Partitioning the edges of $\eps$-regular pairs.}} Now  we partition randomly every pair $G[V_i,V_j]$ into bipartite subgraphs.  We concentrate only on copies of $\cK'$ from graphs $G[\cup_{i\in S}V_i]$ with 
 $\alpha_S\ge\delta/T^r$.  Since there are $\binom{t}{r}$ many $\alpha_S$ we will neglect (using~\eqref{eq:bound_alpha_cK}) less than  
 \begin{equation}\label{eq:sparse_S}
\binom{t}{r}(\delta/T^r)\lceil n/t\rceil^2\le \delta n^2
\end{equation} edge-disjoint copies from $\cK'$. 
  For each $ij\in E(R)$, we split the edges within each $\eps$-regular pair $(V_i,V_j)$ randomly with probabilities $\alpha_S/d_{ij}$, where $i,j\in S\in\binom{V(R)}{r}$, as follows. 
  
  For every edge $e\in E(V_i,V_j)$ we consider the random variable $Y_{ij,e}$, which takes values in $\{S\colon S\in\binom{V({R})}{r}, i,j\in S\}$ with 
probability $\PP(Y_{ij,e}=S)=\alpha_S/d_{ij}$ (and possibly some arbitrary other value with probability $1-\sum_{S\in\binom{V({R})}{r}, i,j\in S} \alpha_S/d_{ij}$). Thus, for a fixed $S\in\binom{V(R)}{r}$, the $2$-set $\{i,j\}\subseteq S$ and an  edge $e\in E(V_i,V_j)$,  the indicator random variable $Z_{S,ij,e}:=1_{\{Y_{ij,e}=S\}}$ is  a Bernoulli variable with parameter $\alpha_S/d_{ij}$. In particular, for fixed $S$, the random variables $Z_{S,ij,e}$ are (mutually) independent.

In this way we obtain, for every $S$, a random $r$-partite subgraph $G_S\subseteq G[\cup_{i\in S} V_i]$ where $\PP(e\in G_S)=\alpha_S/d_{ij}$ for every $e\in E(V_i,V_j)$ for some $i,j\in S$. For given $S\in\binom{V({R})}{r}$ with $\alpha_S\ge\delta/T^r$ and a $2$-set $\{i,j\}\subseteq S$,  we let $X_{ij}$ be the random variable which counts the number of edges from $E(V_i,V_j)$ chosen to be in $G_S$.  
Observe that the $Y_{ij,e}$'s make sure that the graphs $G_S$ are edge-disjoint for different sets $S$. 
Since $X_{ij}$ is the sum of $|E(V_i,V_j)|=d_{ij}|V_i||V_j|$ independent indicator random variables, which are distributed $\mathrm{Be}(\alpha_S/d_{ij})$ we have $\EE(X_{ij})=\alpha_S|V_i||V_j|\ge (\delta/T^r) \lfloor n/t\rfloor^2$, and, by Chernoff's inequality (Theorem~\ref{thm:Chernoff}):
\begin{equation}\label{eq:Chernoff}
 \PP [|X_{ij}  -\EE(X_{ij}) |\ge \xi \EE(X_{ij})] \le 2\exp\left(-\xi^2 \delta n^2/(4t^2T^r)\right).
\end{equation}
Thus, with probability at least $1-2\binom{r}{2}\binom{T}{r}\exp\left(-\xi^2 \delta n^2/(4t^2T^r)\right)=1-o(1)$,  
for every $S$ with $\alpha_S\ge\delta/T^r$ and for all $i\neq j\in S$, we have that $|E_{G_S}(V_i,V_j)|=(1\pm\xi)\alpha_S|V_i||V_j|$ and  the density of every pair $(V_i,V_j)$ in $G_S$ is  thus $(1\pm \xi)\alpha_S$. 

\textbf{\emph{Putting everything together.}}
To conclude the lemma we need the following claim, whose proof we postpone first.  
\begin{claim}\label{claim:PS}
With probability $1-o(1)$, we have that for every $S\in\binom{V({R})}{r}$, where $\alpha_S\ge\delta/T^r$,  the graph $G_S$ contains at least $(1-3\xi)\alpha_S (n/t)^2$ edge-disjoint rainbow copies of $K_r$.
\end{claim}

Since for every $S\in\binom{V({R})}{r}$  the graphs $G_S$ are edge-disjoint, we find at least
\[
\sum_{S\colon \alpha_S\ge \delta/T^r} (1-3\xi)\alpha_S (n/t)^2
\overset{\eqref{eq:bound_alpha_cK}}{\ge} 
(1-4\xi)\sum_{S\colon \alpha_S\ge \delta/T^r} |\cK'(S)|
\ge (1-4\xi)(|\cK'|-\delta n^2)
\]
edge-disjoint rainbow copies of $K_r$. The last inequality 
follows since at most $\delta n^2$ edges lie in $G_S$ with small 
$\alpha_S<\delta/T^r$, cf.~\eqref{eq:sparse_S}.

Thus, we have $\nr(G,K_r)\ge (1-4\xi)(|\cK|-\delta n^2)-\delta n^2\ge |\cK|-3\delta n^2$. 
It follows that
\[
N(G,K_r)\le \nr(G,K_r)+3\delta n^2\le\nr(G,K_r)+c n^2.
\]
\end{proof}

\begin{proof}[Proof of Claim~\ref{claim:PS}]
We fix some $S\in\binom{V({R})}{r}$ with $\alpha_S\ge \delta/T^r$. 
We define the auxiliary $\binom{r}{2}$-uniform hypergraph $\ccH:=\ccH_S$ with the vertex set $V(\ccH_S)=E(G_S)$ as 
follows. Its hyperedges correspond to the edge-sets of all rainbow copies of $K_r$ in $G_S$. The number of edges of $G_S$ and 
thus the vertices of 
$\ccH$ is $(1\pm2\xi)\binom{r}{2}\alpha_S(n/t)^2$ with probability 
$1-2r^2\exp\left(-\xi^2 \delta n^2/(4t^2T^r)\right)=1-o(1)$, see the application of Chernoff's inequality~\eqref{eq:Chernoff}. 

Next we estimate the number of hyperedges of $\ccH$ and related quantities. By Lemma~\ref{lem:most_rainbow}, 
for $i\neq j\in S$,
 all but at most $4r\eps\lceil n/t\rceil^2$ edges from $(V_i,V_j)$ lie in 
 \begin{equation*}
(1\pm\gamma)\left(\prod_{\substack{\{k,\ell\}\in S\\
\{k,\ell\}\neq\{i,j\}}} d_{k\ell}\right)\prod_{a\in S,a\neq i,j}|V_a|
\end{equation*}
rainbow copies of $K_r$ between the sets $V_\ell$, $\ell\in S$. We refer to these edges as good. The remaining at most $4r\eps\lceil n/t\rceil^2$ \emph{bad} edges from $(V_i,V_j)$ lie in at most $\lceil n/t\rceil^{r-2}$ rainbow copies of $K_r$.  

We may now apply Lemma~\ref{lem:sparsify} (with $p_{ij}=\alpha_S/d_{ij}$ for all $i\neq j\in S$, 
$\gamma_{L\ref{lem:sparsify}}=1/2$ and $\eta=\gamma$ obtaining the parameter $\beta$) to these (rainbow) copies to conclude that, with probability $1-n^2e^{-\lfloor n/t\rfloor^{\beta}}$, each good edge closes in $G_S$
\[
(1\pm\gamma)\prod_{\substack{\{k,\ell\}\in S\\
\{k,\ell\}\neq\{i,j\}}} \frac{\alpha_S}{d_{k\ell}}
(1\pm\gamma)
\left(\prod_{\substack{\{k,\ell\}\in S\\
\{k,\ell\}\neq\{i,j\}}} d_{k\ell}\right)\prod_{a\in S, a\neq i,j}|V_a|=
(1\pm2\gamma)\alpha_S^{\binom{r}{2}-1}(n/t)^{r-2}
\]
rainbow copies (which are thus hyperedges in $\ccH$). 

Recall that every bad edge from $E(V_i,V_j)$ is included in $G_S$ with probability $\alpha_S/d_{ij}$ independently. 
 We denote by $B_{ij}$ the number of bad edges from $E(V_i,V_j)$ that are included in $G_S$. An application of Chernoff's inequality guarantees that 
\[
 \PP [B_{ij} \ge  7(\alpha_S/d_{ij})4r\eps\lceil n/t\rceil^2] \le \exp\left(-(\alpha_S/d_{ij})4r\eps\lceil n/t\rceil^2\right),
\]
where we used $\EE(B_{ij})\le (\alpha_S/d_{ij})4r\eps\lceil n/t\rceil^2$. This implies that the above estimate holds for each $\alpha_S\ge \delta/T^r$ and every $\{i,j\}\in \binom{S}{2}$ with probability at most $\binom{r}{2}\binom{T}{r}\exp\left(-(\alpha_S/\delta)4r\eps\lceil n/t\rceil^2\right)$.

Further we need to bound the number of rainbow copies that a bad edge from $G_S$ closes. Now we apply Lemma~\ref{lem:sparsify} (with $p_{ij}=\alpha_S/d_{ij}$ for all $i\neq j\in S$, 
$\gamma_{L\ref{lem:sparsify}}=1/2$ and $\eta=1/2$ obtaining the parameter $\beta'$) to the at most $\lceil n/t\rceil ^{r-2}$ rainbow copies (we can extend these to exactly this number $\lceil n/t\rceil ^{r-2}$ by adding some arbitrary copies of $K_r$).  We  conclude that, with probability $1-n^2e^{-\lfloor n/t\rfloor^{\beta'}}$, each bad edge closes in $G_S$ at most
\[
\frac{3}{2}\left(\prod_{\substack{\{k,\ell\}\in S\\
\{k,\ell\}\neq\{i,j\}}} \frac{\alpha_S}{d_{k\ell}}\right) \lceil n/t\rceil ^{r-2}
\le
2(\alpha_S/\delta)^{\binom{r}{2}-1}(n/t)^{r-2}
\]
rainbow copies (which are thus hyperedges in $\ccH$).  

To summarize: with probability at least
\begin{multline*}
1-r^2T^r\exp\left(-\xi^2 \delta n^2/(4t^2T^r)\right)- r^2T^r\exp\left(-(\alpha_S/\delta)4r\eps\lceil n/t\rceil^2\right)\\
-n^2 T^re^{-\lfloor n/t\rfloor^{\beta}}-n^2 T^re^{-\lfloor n/t\rfloor^{\beta'}}=1-o(1),
\end{multline*}
we have the following for every $S$ with $\alpha_S\ge \delta/T^r$:
\begin{itemize}
\item the number of edges of $G_S$ is
 $(1\pm2\xi)\binom{r}{2}\alpha_S(n/t)^2$, and
\item each good edge (for some pair $(V_i,V_j)$ with $d_{ij}\ge \delta$) closes in $G_S$ 
\[
(1\pm2\gamma)\alpha_S^{\binom{r}{2}-1}(n/t)^{r-2}
\] rainbow copies of $K_r$, and 
\item the number of bad edges in $G_S$ is at most 
\[7\binom{r}{2}(\alpha_S/d_{ij})4r\eps\lceil n/t\rceil^2\le 7\binom{r}{2}(\alpha_S/\delta)4r\eps\lceil n/t\rceil^2,\]
and
\item each bad edge closes in $G_S$ at most
$2(\alpha_S/\delta)^{\binom{r}{2}-1}(n/t)^{r-2}$ rainbow copies of $K_r$.
\end{itemize}

 This means that all but at most $\binom{r}{2}(\alpha_S/\delta)4r\eps\lceil n/t\rceil^2\le c_1 |V(\ccH)|$ vertices  $e\in E(G_S)=V(\ccH)$ have degree in the interval 
\[
(1\pm2\gamma)\alpha_S^{\binom{r}{2}-1}(n/t)^{r-2}.
\]
Thus,  if we set $D:=\alpha_S^{\binom{r}{2}-1}(n/t)^{r-2}$ then it is larger 
 than $d_0\left(\binom{r}{2},c_0, k\right)$ since $r\ge 3$. On the other hand,
 the degree of every vertex of $\ccH$ is at most $2\delta^{-\binom{r}{2}+1}D$.
Furthermore,   any two edges in $E(G_S)$ lie in at most $\lceil n/t\rceil^{r-3}=o(D)$ hyperedges. 
Thus, the assumptions of 
  Theorem~\ref{PS-matching}  are verified and hence  $\ccH$ has a 
matching with at least 
\[
(1-c_0)e(G_S)/\binom{r}{2}\ge (1-c_0)(1-2\xi)\alpha_S (n/t)^2\ge (1-3\xi)\alpha_S (n/t)^2
\]
 hyperedges. 

We conclude that,  with probability $1-o(1)$, 
every graph $G_S$ (with $\alpha_S\ge \delta/T^r$) contains at least $(1-3\xi)\alpha_S (n/t)^2$ edge-disjoint rainbow copies of $K_r$. 
\end{proof}

\subsection{Proof of Theorem~\ref{thm:stab_phi_rainbow}}
\begin{proof}[Proof of Theorem~\ref{thm:stab_phi_rainbow}]
For given $r\ge 3$ and $\gamma>0$ let $\eps'=\eps_{T\ref{thm:stab_phi_cliques}}$ be as asserted by Proposition~\ref{thm:stab_phi_cliques}. 
We choose $c:=\eps'/e(K_r)$ and assume that $n$ is large enough so that Lemma~\ref{lem:almost_edge_rainbow} is 
applicable. Finally we set $\eps:=\eps'/e(K_r)$. 

Let now $G$ be a graph on $n$ vertices with a proper coloring $\chi$ of its edges such that $\pr(G,K_r)\ge\ex(n,K_r)-\eps n^2$ holds. 
By Lemma~\ref{lem:almost_edge_rainbow}, we have $N(G,K_r)\le \nr(G,K_r)+c n^2$ (for $n$ large enough). Therefore, 
$\pr(G,K_r)=e(G)-(e(K_r)-1)\nr(G,K_r)\ge\ex(n,K_r)-\eps n^2$  implies 
\begin{multline*}
\phi(G,K_r)=e(G)-(e(K_r)-1)N(G,K_r)\ge e(G)-(e(K_r)-1)(\nr(G,K_r)+c n^2)\ge\\
\pr(G,K_r)-(e(K_r)-1)cn^2
\ge
\ex(n,K_r)-\eps n^2-(e(K_r)-1)cn^2=\\ \ex(n,K_r)-e(K_r)cn^2=\ex(n,K_r)-\eps'n^2,
\end{multline*}
 and Proposition~\ref{thm:stab_phi_cliques} yields the desired partition.
\end{proof}

\section{Proof of Theorem~\ref{thm:main-rainbow}}\label{sec:exact}
In the following, we conclude with the proof of Theorem~\ref{thm:main-rainbow}, 
which follows by using the same steps (Claims 7-9) of the proof 
of the main result from~\cite[Theorem~3]{LO-YP} except a minor modification at the end which we are going to describe below. 
Although the main theorem in~\cite[Theorem~3]{LO-YP} did not consider the case $H=K_r$, 
the steps of the proof work verbatim for this case as well.   
\begin{proof}[Sketch of the proof of Theorem~\ref{thm:main-rainbow}]
We will apply Theorem~\ref{thm:stab_phi_rainbow} 
in the  form when $\eps=0$, and 
we choose $\gamma$ sufficiently small. We assume that there is a graph $G$ on $n$ vertices ($n$ large enough) such that there exists a proper edge coloring of $G$ with $\pr(G,K_r)\ge \ex(n,K_r)$  and $G$ is not isomorphic to the Tur\'an graph $T_{r-1}(n)$. 
By following the steps of~\cite{LO-YP}, we apply first Claim 7 from~\cite{LO-YP} and 
assume that $\pr(G) = \ex(n,K_r) + m$ for some $m\geq 0$.  
We obtain a subgraph $G'$ of $G$ on $n'$ vertices such that 
$\delta(G')\ge\delta(T_{r-1}(n'))$ and $\pr(G',K_r)\ge  \ex(n',K_r)+m$.  
Then, Theorem~\ref{thm:stab_phi_rainbow}  asserts the existence of a partition of $G'$ into $r-1$ parts that maximizes 
the number of edges between different parts so that the number $m_2$ of edges  within the partition classes is not zero but  is at most $\gamma n'^2$. 
It is observed that the order of these $r-1$ parts are almost balanced due 
to the maximality condition (Claim~8 from~\cite[Theorem~3]{LO-YP}).  This implies that $e(G')\le\ex(n',K_r)+m_2$. 

In the final step, we find at least $\left\lfloor\frac{m_2}{\binom{r}{2}-1}\right \rfloor+1$  many rainbow copies of $K_r$.  
We find these copies iteratively (Claim~9 from~\cite[Theorem~3]{LO-YP}). 
The only difference to  the embedding in~\cite{LO-YP} is that we need to find at each iteration 
a copy of $K_r$, which under the edge coloring $\chi$ is rainbow. This is done by finding first a complete $(r-1)$-partite graph $F$ with parts of size at least $8(r-1)^3$, so that one of the parts contains an edge (Claim~9 from~\cite[Theorem~3]{LO-YP}, this basically follows from an application of Theorem of Erd\H{o}s and Stone~\cite{ES46}). A rainbow copy of $K_r$ is guaranteed to exist in such a graph $F$ by a result of 
Keevash, Mubayi, Sudakov, and  Verstra{\"e}te~\cite[Lemma 2.2]{KMSV07}. This allows us to 
 find at least $\left\lfloor\frac{m_2}{\binom{r}{2}-1}\right \rfloor+1$  many rainbow copies of $K_r$ in $G'$ under the edge coloring $\chi$. Clearly, this implies that 
 \[
 \pr(G',K_r)\le  \ex(n',K_r)+m_2-\left(\binom{r}{2}-1\right)\left(\left\lfloor\frac{m_2}{\binom{r}{2}-1}\right \rfloor+1\right)<\ex(n,K_r),
 \]
 a contradiction.
\end{proof}

\section{Concluding remarks}\label{sec:remarks}
With essentially the same techniques, Theorem~\ref{thm:main-rainbow} can be generalized to any 
edge-critical graph of chromatic number at least $3$ as follows.
\begin{theorem}\label{thm:edge_critical}
For any edge-critical graph $H$ with chromatic number at least 3, 
there is an $n_0=n_0(H)$ such that $\prn(n,H)=\ex(n,H)$ for all $n\ge n_0$. 
Moreover, the only graph attaining $\prn(n,H)$ is the Tur\'an graph $T_{\chi(H)-1}(n)$.
\end{theorem}  
To do so one can prove the following stability result about $\pr(n,H)$.
 \begin{theorem}\label{thm:general_stability}
 For every $H$ with $\chi(H)=r\ge 3$ and for every $\gamma>0$ there exist $\eps>0$ and $n_0\in\Nat$ such that for every graph $G$ on $n\ge n_0$ vertices with a proper coloring $\chi$ of its edges the following is true. If 
\[
\pr(G,H)\ge\ex(n,H)-\eps n^2 
\]
then there exists a partition of $V(G)=V_1\dcup\ldots\dcup V_{r-1}$ with $\sum_{i=1}^{r-1}e(V_i)<\gamma n^2$.
\end{theorem}
Theorem~\ref{thm:general_stability} can be shown by generalizing the proof of Theorem~\ref{thm:stab_phi_rainbow}. 
We provide here a very short sketch of a slightly different approach, which generalizes~\cite[Lemma~4]{LO-YP}. We would like to stress however that this does not apply to cliques  though. 
\begin{proof}[Proof sketch of Theorem~\ref{thm:general_stability}]
The idea is again based on the regularity lemma,  on the result of Pikhurko and Sousa~\cite[Theorem~1.1]{PS07} 
and an application of Theorem~\ref{PS-matching}. 

We first fix an arbitrary proper edge coloring $\chi$ of $G$. 
The proof follows along the lines of the argument in~\cite[Lemma~4]{LO-YP}, up to the place when the auxiliary $e(H)$-uniform  hypergraph is built (which is done in~\cite[Corollary~6]{LO-YP}). The only difference is that our hypergraph consists now of rainbow copies of $H$ (and these are the most copies of $H$ by an argument similar to Lemma~\ref{lem:most_rainbow}). The remainder of the proof stays the same.
\end{proof}
 
As a direct consequence of Theorem~\ref{thm:general_stability} we obtain the following generalization of the result 
of Pikhurko and Sousa~\cite{PS07} 
on $\phi (n, H)$ mentioned above.

\begin{theorem}\label{thm:main-rb-gen}
Let $H$ be any graph of chromatic number at least three. Then we have
\[
\prn(n,H)=\ex(n,H)+o(n^2).
\]
\end{theorem}
\begin{proof}
The lower bound follows by taking any proper coloring of any $H$-extremal graph on $n$ vertices.
 The upper bound follows since \emph{any} proper  edge coloring  $\chi$ of any graph $G$ on $n$ vertices with 
 $\pr(G,H)\ge\ex(n,H)$ admits  a partition of $V(G)=V_1\dcup\ldots\dcup V_{r-1}$ with $\sum_{i=1}^{r-1}e(V_i)=o(n^2)$, implying $e(G)\le \ex(n,H)+o(n^2)$ and thus $\prn(n,H)= \ex(n,H)+o(n^2)$.
\end{proof}
Finally, we would like to mention that Theorems~\ref{thm:edge_critical} and~\ref{thm:main-rb-gen} can be stated in a slightly general form, similar to the one of Theorem~\ref{thm:main-rainbow}.

\bibliographystyle{amsplain}
\bibliography{decomps}

\end{document}